\documentclass[11pt]{article}
\usepackage{latexsym,bm}
\usepackage{mathrsfs}
\usepackage{amsmath,amsthm}
\usepackage{graphicx}
\usepackage{amssymb}
\usepackage{comment}
\usepackage{CJK}
\usepackage{color}
\usepackage{cite}
\usepackage{array}
\usepackage{stfloats}
\usepackage[colorlinks,
            linkcolor=blue,       
            anchorcolor=blue,  
            citecolor=blue,        
            ]{hyperref}
\usepackage{caption}
\usepackage{graphicx}
\usepackage{epstopdf}
\usepackage{upgreek}

\newtheorem{theorem}{Theorem}[section]
\newtheorem{lemma}[theorem]{Lemma}

\newtheorem{corollary}[theorem]{Corollary}
\newtheorem{proposition}[theorem]{Proposition}

\newtheorem{problem}[theorem]{Problem}
\theoremstyle{definition}
\newtheorem{definition}[theorem]{Definition}
\newtheorem{example}[theorem]{Example}
\theoremstyle{remark}
\newtheorem{remark}[theorem]{Remark}
\DeclareGraphicsExtensions{.eps,.eps.gz}
\usepackage[top=2.5cm,bottom=2.5cm,left=2.8cm,right=2.2cm]{geometry}
\normalsize \rm
\allowdisplaybreaks[4]
\makeatletter
\@addtoreset{equation}{section}
\makeatother

 \usepackage{indentfirst}
\newcommand{\idim}{\operatorname{idim}}

\begin{document}
\begin{CJK}{GBK}{song}
\newcommand{\song}{\CJKfamily{song}}    
\newcommand{\fs}{\CJKfamily{fs}}        
\newcommand{\kai}{\CJKfamily{kai}}      
\newcommand{\hei}{\CJKfamily{hei}}      
\newcommand{\li}{\CJKfamily{li}}        
\renewcommand\figurename{Fig.}

\begin{center}
{{\huge Resolving the Klav\v{z}ar-Kov\v{s}e conjecture on opposite semicube isomorphisms in partial cubes and its extension}} \\[18pt]
{\Large Zhaoman Huang$^{1}$, Yan-Ting Xie$^{1}$ \footnotetext{*Corresponding author.\\ E-mail address: shjxu@lzu.edu.cn (S.-J. Xu).} Shou-Jun Xu$^{1,*}$ }\\[6pt]
{ \footnotesize  $^{1}$ School of Mathematics and Statistics, Gansu Center for Applied Mathematics, Lanzhou University, Lanzhou, Gansu 730000, China}

\end{center}
\vspace{1mm}
\begin{abstract}

Partial cubes are a fundamental class of graphs that admit isometric embeddings into hypercubes. Klav\v{z}ar and Kov\v{s}e [Ars Combin. 93 (2009), 77--86] observed that the opposite semicubes of every harmonic-even partial cube are pairwise isomorphic, and asked whether the converse is true, that is, whether a partial cube is harmonic-even if and only if its opposite semicubes are pairwise isomorphic. In this paper, we answer this question in the negative by constructing an infinite family of partial cubes with pairwise isomorphic opposite semicubes that are not harmonic-even. This establishes that pairwise opposite-semicube isomorphism is strictly weaker than harmonic-evenness and naturally leads to the question of what additional condition restores the equivalence. To address this question, we introduce the opposite-semicube Helly property and prove that a finite partial cube satisfying this property is antipodal, or equivalently harmonic-even by Polat's theorem, if and only if it has pairwise isomorphic opposite semicubes.

\end{abstract}

\section{Introduction}

Partial cubes form a fundamental class of graphs that admit isometric embeddings into hypercubes. Graham and Pollak~\cite{Graham1971} introduced them in their study of the addressing problem for loop switching. Since then, partial cubes have been studied extensively because of their deep connections with metric graph theory, combinatorial geometry, and network design. Classical results and comprehensive overviews can be found in the monographs \cite{HammackImrichKlavzar2011, Imrich2000} and the survey \cite{Ovchinnikov2008}.

A central theme in the theory of partial cubes is the relationship between local symmetry properties and global antipodal structures, such as distance-balancedness or isomorphism of opposite semicubes. For a connected graph $G$, a vertex $y$ is called a \emph{relative antipode} of $x$ if $d_G(x,y) \ge d_G(x,z)$ holds for every neighbor $z$ of $y$, where $d_{G}$ denotes the usual distance in $G$; equivalently, $y$ is an endpoint of a maximal geodesic starting at $x$  \cite{Sabidussi1996}. A vertex $y$ is said to be an \emph{absolute antipode} of $x$ if $d_{G}(x,y)$ equals the diameter of $G$, where the diameter is defined as the maximum distance between any two vertices in $G$. A graph is \emph{antipodal} if every vertex has exactly one relative antipode. Note that, if $G$ is antipodal, then the unique relative antipode of a vertex $x$ is an absolute antipode
of $x$, and thus is denoted by $\bar{x}$. A related notion is that of an \emph{even graph}, in which every vertex $x$ has a unique absolute antipode $\bar{x}$. A graph is \emph{harmonic-even} (also called \emph{automorphically diametrical} \cite{Sabidussi1996}) if it is even and the antipodal map $x \mapsto \bar{x}$ for $x\in V(G)$ is an automorphism of $G$. It is easy to see that every antipodal graph is harmonic-even. A graph $G$ is \emph{distance-balanced} if for every edge $uv\in E(G)$, the number of vertices closer to $u$ than to $v$ equals the number of vertices closer to $v$ than to $u$. Every harmonic-even graph is distance-balanced, but the converse is false even for partial cubes \cite{Handa1999}.

Antipodal bipartite graphs were first studied by Kotzig \cite{Kotzig1968}, who called such graphs \(S\)-graphs. Subsequent work by Glivjak, Kotzig, and Plesn\'{\i}k~\cite{Glivjak1970} gave a metric characterization. In the setting of partial cubes, Fukuda and Handa~\cite{FukudaHanda1993} conjectured that every harmonic-even partial cube is antipodal. This conjecture was confirmed by Polat \cite{Polat2019}, who proved that, for partial cubes, harmonic-evenness and antipodality are equivalent. He also gave several equivalent characterizations of antipodal partial cubes, involving \(\Theta\)-faithful automorphisms, semicube-switching automorphisms, and the pre-hull number.

Klav\v{z}ar and Kov\v{s}e \cite[Problem~5.3]{KlavzarKovse2009} observed that, in every harmonic-even partial cube, the two opposite semicubes associated with each $\Theta$-class are isomorphic, and asked whether the converse holds. That is, they asked whether a partial cube is harmonic-even if and only if its opposite semicubes are pairwise isomorphic. We call a partial cube with the latter property \emph{opposite-semicube-isomorphic}.

In this paper, we first give a negative answer to the question of Klav\v{z}ar and Kov\v{s}e. For every $n\ge 5$, set $k=n-5$. Define
\[
S_n=\{000000^k,\;000010^k,\;000100^k,\;001010^k,\;110111^k,\;111001^k,\;111101^k,\;111111^k\}\subseteq V(Q_n),
\]
and let $B_n=Q_n-S_n$. The graph $B_5$ is a $24$-vertex graph adapted from an example of Handa \cite{Handa1999}. We prove that $B_n$ is a partial cube, and that for every $\Theta$-class of $B_n$, the two opposite semicubes are isomorphic. However, $B_n$ is not antipodal, and hence not harmonic-even by Polat's theorem. Thus, the local condition of pairwise isomorphic opposite semicubes is strictly weaker than harmonic-evenness. We also show that Cartesian products preserve the relevant local isomorphism property. In particular, $B_n\square Q_m$ is opposite-semicube-isomorphic but not antipodal for every $n\geq5$ and $m\geq0$. These constructions also provide negative answers to the related questions considered by Polat \cite{Polat2019}.

The counterexamples reveal the obstruction to the converse. Opposite-semicube isomorphism concerns one $\Theta$-class at a time, whereas antipodality requires compatibility among all $\Theta$-classes. More precisely, for a vertex $x$, an antipodal vertex $\bar{x}$ must lie in the semicube opposite to $x$ with respect to every $\Theta$-class. Thus, the family of all semicubes opposite to $x$ must have a common vertex. The isomorphism of the two semicubes associated with each individual $\Theta$-class does not, by itself, guarantee such a common intersection. This observation leads naturally to a Helly-type condition.

Helly-type properties play a central role in graph convexity. A standard example comes from median graphs, the convex subsets of a median graph satisfy the finite Helly property; equivalently, convex sets in median graphs are gated, and gated sets form a Helly family \cite{BandeltChepoi2008}. This is also closely related to the classical connection between median graphs and Helly hypergraphs \cite{MulderSchrijver1979}. Since every semicube of a partial cube is convex, this suggests that a Helly-type condition on suitable families of opposite semicubes is a natural condition to consider. Motivated by this, we introduce a weaker and more targeted condition, called the opposite-semicube Helly property. For a fixed vertex \(x\), consider the family of semicubes opposite to \(x\), one for each \(\Theta\)-class. We require only that, whenever this special family is pairwise intersecting, their total intersection is nonempty. This is not the usual Helly property for all convex sets, nor even for all semicubes; rather, it is a condition designed to test whether the local opposite-semicube data can be realized by a single vertex opposite to \(x\) in all \(\Theta\)-classes.

Another contribution of this paper is to show that the above Helly-type condition is sufficient to turn local opposite-semicube isomorphisms into a global antipodal structure. More precisely, we prove that a finite partial cube with the opposite-semicube Helly property is antipodal if and only if its opposite semicubes are pairwise isomorphic. By Polat's theorem, the same characterization holds with harmonic-evenness in place of antipodality. The proof explains the role of the Helly condition. Opposite-semicube isomorphism implies a two-dimensional balance relation for every pair of \(\Theta\)-classes, which in turn gives pairwise intersection of the semicubes opposite to a fixed vertex. The opposite-semicube Helly property then upgrades this pairwise intersection to a common vertex, and this vertex is forced to be the unique relative antipode.

The paper is organized as follows. After recalling the basic terminology, Section~2 introduces the opposite-semicube Helly property. In Section~3, we construct the counterexample family \(B_n=Q_n-S_n\) and show that these graphs are opposite-semicube-isomorphic but not antipodal. Cartesian product extensions of this construction are developed in Section~4. The positive result under the Helly-type condition is established in Section~5. Finally, Section~6 presents concluding remarks and several open problems.

\section{Preliminaries}

All graphs considered are undirected, connected, simple, and finite. For a graph $G$, we denote by $V(G)$ its vertex set and by $E(G)$ its edge set. For a subset $S \subseteq V(G)$, $G[S]$ denotes the subgraph induced by $S$, and $G-S = G[V(G)\setminus S]$. A \emph{path} $P$ of length $n$ is the graph with distinct vertices $x_0, x_1, \dots, x_n$ and edges $x_i x_{i+1}$ ($0\leq i\leq n-1$), denoted by $\langle x_0, x_1, \dots, x_n\rangle$. For $x,y\in V(G)$, the \emph{distance} $d_G(x,y)$ is the length of a shortest $x$,$y$-path in $G$. A shortest $x$,$y$-path of $G$ is called an \emph{$x$,$y$-geodesic} of $G$. The {\em interval} $I_G(x,y)$ of $G$ consists of all vertices that lie on some $x$,$y$-geodesic. A subgraph $H$ of $G$ is \emph{isometric} if $d_H(x,y)=d_G(x,y)$ for all $x,y\in V(H)$, and it is \emph{convex} if for any $x,y\in V(H)$ every $x$,$y$-geodesic in $G$ is contained in $H$. The \emph{Cartesian product} $G \,\square\, H$ of two graphs $G$ and $H$ has vertex set $V(G)\times V(H)$, and $(g,h)$ is adjacent to $(g',h')$ if either $gg'\in E(G)$ and $h=h'$, or $g=g'$ and $hh'\in E(H)$. The distance between $(g,h)$ and $(g',h')$ in $G \,\square\, H$ is consistent with the distance formula of the Cartesian product: $d_{G\,\square\, H}((g,h),(g',h'))=d_G(g,g')+d_H(h,h')$.

For $n\ge 1$, the $n$-dimensional hypercube $Q_n$ is the graph with vertex set $\{0,1\}^n$, where two vertices are adjacent if they differ in exactly one coordinate. We write a vertex $x\in V(Q_n)$ as a binary word $x=x_1x_2\cdots x_n$. Its distance is the Hamming distance $d_{Q_n}(x,y)=|\{i\in \{1, \ldots, n\}: x_i\neq y_i\}|$. Denote \(x^{(i)}\) the vertex obtained from \(x\) by changing the \(i\)-th coordinate for \(i\in\{1,\ldots,n\}\) and \(x^{(i,j)}\)  the vertex obtained by changing both the \(i\)-th and \(j\)-th coordinates for \(i\neq j\) and \(i,j\in\{1,\ldots,n\}\). To avoid confusion with this coordinate-flipping notation, we use superscripts without parentheses only to abbreviate repeated bits in binary strings; for example, \(1^n\) denotes the string \(111\ldots 11\) of length \(n\).
 An induced subgraph $H$ of $Q_n$ is a \emph{partial cube} if it is isometric in $Q_n$.

The \emph{Djokovi\'{c}-Winkler relation} $\Theta$ on $E(G)$ (see \cite{Djokovic1973, Winkler1984}) is defined as follows. For edges $e=uv$ and $f=xy$,
\[
e\,\Theta\,f \iff d_G(u,x)+d_G(v,y) \neq d_G(u,y)+d_G(v,x).
\]
The relation $\Theta$ is clearly reflexive and symmetric. If $G$ is bipartite, this is equivalent to $d_G(u,x)=d_G(v,y)$ and $d_G(u,y)=d_G(v,x)$. Winkler \cite{Winkler1984} proved that a graph is a partial cube if and only if it is bipartite and $\Theta$ is an equivalence relation on $E(G)$. The equivalence classes are called \emph{$\Theta$-classes}, and their number equals the isometric dimension $\idim(G)$ of $G$, where the \emph{isometric dimension} denotes the least
non-negative integer $n$ such that $G$ is an isometric subgraph of $Q_{n}$.

Let \(G\) be a partial cube, and let \(E_1,\ldots,E_d\) be its \(\Theta\)-classes. For each \(\Theta\)-class \(E_i\), the removal of the edges in \(E_i\) separates \(G\) into two convex components, called the two \emph{semicubes} associated with \(E_i\). We denote their vertex sets by $W_i^0~ \text{and} ~W_i^1.$ Equivalently, if \(xy\in E_i\), then, after possibly interchanging the labels \(0\) and \(1\),
\[
W_i^0=\{a\in V(G): d_G(x,a)<d_G(y,a)\},~\text{and}~
W_i^1=\{a\in V(G): d_G(y,a)<d_G(x,a)\}.
\]
The induced subgraphs \(G[W_i^0]\) and \(G[W_i^1]\) are called \emph{opposite semicubes}. A basic property of partial cubes is that every semicube is convex in \(G\) \cite{Djokovic1973}. We say that \(G\) is \emph{opposite-semicube-isomorphic} if $G[W_i^0]\cong G[W_i^1]$ for every \(i\in\{1,\ldots,d\}\). For \(x\in V(G)\), let \(W_i(x)\) denote the unique semicube among \(W_i^0\) and \(W_i^1\) that contains \(x\), and let $W_i^*(x)$ denote the other semicube. Under a fixed isometric embedding of $G$ into a hypercube, let \(E_i\) be the \(\Theta\)-class corresponding to coordinate \(i\). Then, \(W_i^0\) and \(W_i^1\) are precisely the two coordinate halfspaces $\{x\in V(G)|x_i=0\}~\text{and}~
\{x\in V(G)|x_i=1\}$, up to the choice of labels \(0\) and \(1\). When several graphs are considered simultaneously, we write \(W_i^{G,0}\), \(W_i^{G,1}\), \(W_i^G(x)\), and \(W_i^{G,*}(x)\) to indicate the ambient graph.

We now introduce the Helly-type condition that will be used to assemble the local choices of opposite semicubes into a single global vertex.

\begin{definition}
Let \(G\) be a finite partial cube with \(\Theta\)-classes \(E_1,\ldots,E_d\). For each \(x\in V(G)\), set
\[
\mathcal O_G(x)=\{W_i^{*}(x)\mid 1\le i\le d\}.
\]
We say that \(G\) has the \emph{opposite-semicube Helly property} if, for every \(x\in V(G)\), whenever $W_i^{*}(x)\cap W_j^{*}(x)\neq\emptyset
\ \text{for all } 1\le i,j\le d$ and $i\neq j$, we have $\bigcap_{i=1}^d W_i^{*}(x)\neq\emptyset.$
\end{definition}

\begin{remark}
The opposite-semicube Helly property should be understood as a condition on the local choices of opposite semicubes.  Since every semicube is convex in a partial cube, median graphs---the graphs whose family of all convex subsets satisfies the finite Helly property \cite{BandeltChepoi2008}---are a subclass of graphs that obviously satisfy the opposite-semicube Helly property. Besides, see also \cite{MulderSchrijver1979} for the classical connection between median graphs and Helly hypergraphs.

The condition is not vacuous, i.e., it is not equivalent to antipodality. For example, every tree is a median graph, and hence satisfies the opposite-semicube Helly property, but most trees are not antipodal. Thus, the opposite-semicube Helly property alone does not force antipodality.

This condition plays a different role. It provides the global compatibility needed to assemble the local choices of opposite semicubes into a single vertex that is opposite to a given vertex in every \(\Theta\)-class. The counterexamples constructed in Section~3 show that, without this condition, the property of
opposite-semicube-isomorphism cannot imply antipodality.
\end{remark}


For a subset $A\subseteq V(G)$, its \emph{convex hull} $co_G(A)$ is the smallest convex set containing $A$. A \emph{copoint} at a vertex $x$ is a convex set $C$ maximal with respect to $x\notin C$; the vertex $x$ is called an \emph{attaching point} of $C$. The \emph{pre-hull operator} $I_G$ is defined by $I_G(A)=\bigcup_{x,y\in A} I_G(x,y)$, and the convex hull satisfies $co_G(A)=\bigcup_{n\ge 0} I_G^n(A)$. For a vertex $x\in V(G)$ and a copoint $C$ at $x$, define $r(x;C)$ to be the smallest $n$ such that $co_G(C\cup\{x\})=I_G^n(C\cup\{x\})$; if no such $n$ exists, $r(x;C)=\infty$. The \emph{pre-hull number} $ph(G)$ is defined as the maximum value of $r(x;C)$ among all vertices of $G$ and all copoints at them. An automorphism $\alpha$ of a partial cube $G$ is \emph{$\Theta$-faithful} if for every edge $xy$, the edges $xy$ and $\alpha(y)\alpha(x)$ belong to the same $\Theta$-class. It is \emph{semicube-switching} if it induces an isomorphism $G[W_{i}^0]\cong G[W_{i}^1]$ for every $\Theta$-class $E_{i}$ of $G$. Polat \cite{Polat2019} proved the following characterizations.

\begin{theorem}\cite{Polat2019}\label{T2.1}
For a partial cube $G$, the following are equivalent.
\begin{enumerate}
\item $G$ is antipodal.
\item $G$ has a $\Theta$-faithful automorphism.
\item $G$ has a semicube-switching automorphism.
\item $ph(G)\le 1$ and $G$ is distance-balanced.
\end{enumerate}
Moreover, $G$ is harmonic-even if and only if it is antipodal.
\end{theorem}

It is known that  \(G=H\square K\) is a partial cube if both $H$ and $K$ are partial cubes. We shall use the following standard observation.

\begin{lemma}\label{lem:product-semicubes}
Let \(H\) and \(K\) be partial cubes and \(G=H\square K\). Let \(E_i^H\) be a \(\Theta\)-class of \(H\). The corresponding \(H\)-direction \(\Theta\)-class of \(G\) is
\[
E_{H,i}^G=\{(h,k)(h',k)\mid hh'\in E_i^H,\ k\in V(K)\}.
\]
If \(W_i^{H,0}\) and \(W_i^{H,1}\) are the two semicubes of \(H\) associated with \(E_i^H\), then, with consistent labels, the two semicubes of \(G\) associated with \(E_{H,i}^G\) are
\[
W_{H,i}^{G,\varepsilon}=W_i^{H,\varepsilon}\times V(K)~\text{for}~ \varepsilon\in\{0,1\}.
\]
Similarly, let \(E_j^K\) be a \(\Theta\)-class of \(K\). The corresponding \(K\)-direction \(\Theta\)-class of \(G\) is
\[
E_{K,j}^G=\{(h,k)(h,k')\mid kk'\in E_j^K,\ h\in V(H)\}.
\]
If \(W_j^{K,0}\) and \(W_j^{K,1}\) are the two semicubes of \(K\) associated with \(E_j^K\), then, with consistent labels, the two semicubes of \(G\) associated with \(E_{K,j}^G\) are
\[
W_{K,j}^{G,\varepsilon}=V(H)\times W_j^{K,\varepsilon}~\text{for}~ \varepsilon\in\{0,1\}.
\]
\end{lemma}

\begin{proof}
Let \(E_i^H\) be a \(\Theta\)-class of \(H\). Choose an edge \(hh'\in E_i^H\) such that
\[
W_i^{H,0}=\{u\in V(H)\mid d_H(h,u)<d_H(h',u)\},~ \text{and} ~
W_i^{H,1}=\{u\in V(H)\mid d_H(h',u)<d_H(h,u)\}.
\]
Fix \(k\in V(K)\). For any \((u,v)\in V(G)\), by the distance formula of Cartesian product,
\[
d_G((h,k),(u,v))<d_G((h',k),(u,v))
\]
holds if and only if
\[
d_H(h,u)+d_K(k,v)<d_H(h',u)+d_K(k,v),
\]
which is equivalent to $d_H(h,u)<d_H(h',u).$ Thus, \((u,v)\) lies in the semicube of \(G\) containing \((h,k)\) precisely when \(u\in W_i^{H,0}\). Therefore, this semicube is $W_i^{H,0}\times V(K).$ The opposite semicube is consequently $W_i^{H,1}\times V(K).$ Hence, with consistent labels,
\[
W_{H,i}^{G,\varepsilon}=W_i^{H,\varepsilon}\times V(K)~\text{for}~ \varepsilon\in\{0,1\}.
\]
The proof for a \(\Theta\)-class \(E_j^K\) is analogous.
\end{proof}

\section{Construction of counterexamples}

This section presents constructive counterexamples that resolves the question posed by Klav\v{z}ar and Kov\v{s}e in the negative. The counterexamples are based on a 24-vertex graph \(B_5\) inspired by Handa \cite{Handa1999}, which is the smallest counterexample found in our current search within $Q_5$.

Let
\[
        S_5=\{00000,00001,00010,00101,11011,11100,11110,11111\}
        \subseteq V(Q_5)
\]
and define $B_5=Q_5-S_5.$ Thus, \(|V(B_5)|=24\). The graph $B_{5}$ is shown in Fig.\ref{F1}.

\begin{figure}[htbp]
    \centering
        \includegraphics[width=0.8\textwidth]{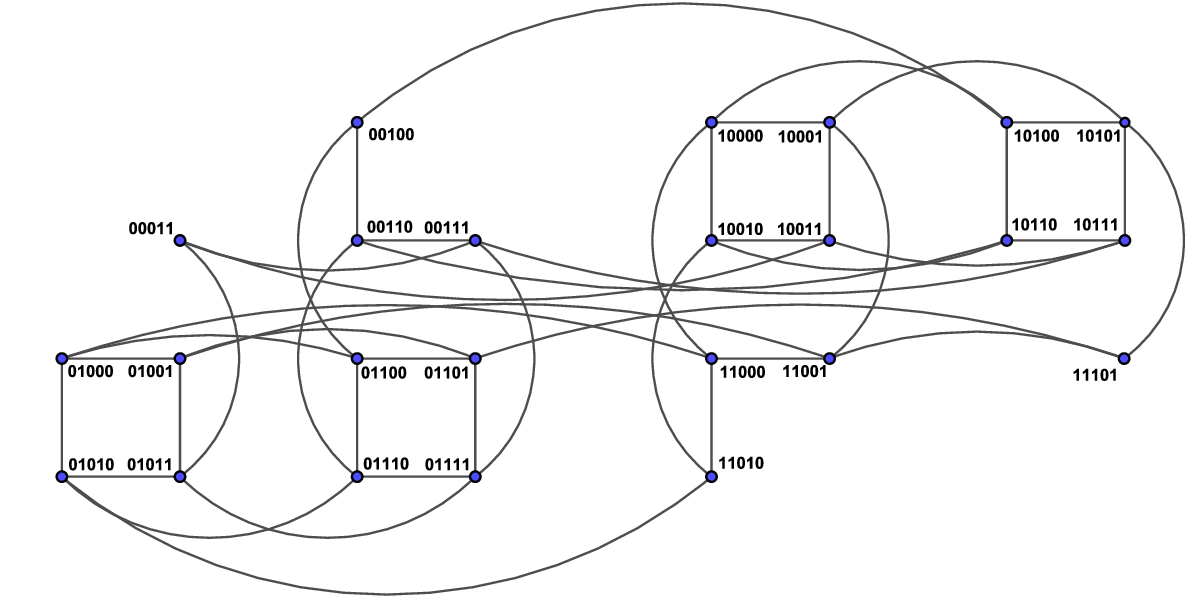}
    \caption{The graph of $B_5$.}
    \label{F1}
\end{figure}

We extend this base 24-vertex graph to an infinite family of counterexamples to the assertion that a partial cube with isomorphic opposite semicubes for every $\Theta$-class must be antipodal or harmonic-even.

For $n\geq 5$, put $k=n-5$.  Define
\[
S_n=\{00000 0^k,00001 0^k,00010 0^k,00101 0^k,
11011 1^k,11100 1^k,11110 1^k,11111 1^k\}\subseteq V(Q_n),
\]
and set $B_n=Q_n-S_n$. When $k=0$, the words $0^k$ and $1^k$ are understood to be empty, i.e., $B_5$ is exactly the 24-vertex Handa-type graph.

We first prove that deleting the eight vertices in $S_n$ does not change the hypercube distances between the remaining vertices.

\begin{lemma}\label{lem:deleted-neighborhood}
Let $n\geq 5$. If a vertex $u\in V(Q_n)\setminus S_n$ is adjacent to two vertices of $S_n$ in $Q_n$, then $u$ is one of the four vertices
\[
00011 0^k, 00100 0^k, 11010 1^k, 11101 1^k.
\]
Moreover, if the two deleted neighbors of $u$ are $u^{(i)}$ and $u^{(j)}$, then $u^{(i,j)}\in S_n$. In particular, every vertex outside $S_n$ has at most two deleted neighbors.
\end{lemma}

\begin{proof}
A vertex of $Q_n$ can be adjacent to two deleted vertices only if those two deleted vertices are at Hamming distance $2$. 
A direct inspection of $S_n$ shows that the unordered pairs of deleted vertices at Hamming distance $2$ are precisely those listed in Table~\ref{tab:common-neighbors}.

\begin{table}[htbp]
\centering
\caption{Common neighbors inside and outside $S_n$.}
\label{tab:common-neighbors}
\begin{tabular}{c|c|c}
\hline\hline
pair in $S_n$ & common neighbor outside $S_n$ & other common neighbor in $S_n$ \\
\hline
$\{00001 0^k,00010 0^k\}$ & $00011 0^k$ & $00000 0^k$ \\
\hline
$\{00000 0^k,00101 0^k\}$ & $00100 0^k$ & $00001 0^k$ \\
\hline
$\{11110 1^k,11011 1^k\}$ & $11010 1^k$ & $11111 1^k$ \\
\hline
$\{11111 1^k,11100 1^k\}$ & $11101 1^k$ & $11110 1^k$\\
\hline\hline
\end{tabular}
\end{table}

All other pairs in $S_n$ have Hamming distance different from $2$. Hence, the only vertices outside $S_n$ adjacent to two deleted vertices are precisely the four vertices displayed in the middle column. For each row, the vertex listed in the final column is obtained by changing the two corresponding coordinates of the vertex in the middle column. This proves the assertion.
\end{proof}

\begin{proposition}\label{prop:isometric}
For every $n\geq 5$, the graph $B_n$ is a partial cube.
\end{proposition}

\begin{proof}
Let $x,y\in V(B_n)$. We prove by induction on $r=d_{Q_n}(x,y)$ that there is an $x,y$-path of length $r$ in $B_n$. The case $r=0$ is trivial.

Assume $r>0$ and put $R=\{i\in\{1,\ldots,n\}:x_i\neq y_i\}.$ Now we show that there exists some $i\in R$ such that $x^{(i)}\notin S_n$. Suppose not. Then, every vertex $x^{(i)}$, $i\in R$, is deleted. By Lemma~\ref{lem:deleted-neighborhood}, the vertex $x\notin S_n$ has at most two deleted neighbors. Hence, $|R|\leq 2$. If $|R|=1$, say $R=\{i\}$, then $y=x^{(i)}\in S_n$, contradicting $y\in V(B_n)$. If $|R|=2$, say $R=\{i,j\}$, then $x^{(i)},x^{(j)}\in S_n$, and Lemma~\ref{lem:deleted-neighborhood} gives $x^{(i,j)}\in S_n$. But $y=x^{(i,j)}$, again a contradiction. Therefore, there is some $i\in R$ such that $x^{(i)}\notin S_n$.

Pick an $i\in R$ such that $x^{(i)}\in V(B_n)$. Then, $d_{Q_n}(x^{(i)},y)=r-1.$ By the induction hypothesis, there is an $x^{(i)},y$-path of length $r-1$ in $B_n$. Adding the edge $xx^{(i)}$ gives an $x,y$-path of length $r$ in $B_n$. Therefore, $d_{B_n}(x,y)\leq r=d_{Q_n}(x,y).$ The reverse inequality is trivial because $B_n$ is an induced subgraph of $Q_n$. Hence, $d_{B_n}(x,y)=d_{Q_n}(x,y)$ for all $x,y\in V(B_n)$, and $B_n$ is isometric in $Q_n$.
\end{proof}

We next show that every pair of opposite semicubes of $B_n$ is isomorphic. The proof uses explicit hypercube automorphisms preserving the deletion set.

\begin{proposition}\label{prop:semicubes}
For every \(n\geq 5\) and every \(i\in\{1,\ldots,n\}\), $B_n[W_i^{0}]\cong B_n[W_i^{1}].$
\end{proposition}

\begin{proof}
By Proposition~\ref{prop:isometric}, \(B_n\) is an isometric subgraph of \(Q_n\). Hence, the \(\Theta\)-classes of \(B_n\) are inherited from the coordinate directions of \(Q_n\). We label them so that the \(\Theta\)-class \(E_i\) consists of the edges of \(B_n\) whose endpoints differ in the \(i\)-th coordinate. With this labeling, the two semicubes of \(B_n\) associated with \(E_i\) are precisely
\[
W_i^{0}=V(B_n)\cap\{x_i=0\}
~\text{and}~
W_i^{1}=V(B_n)\cap\{x_i=1\}.
\]

Now, we prove that, for every coordinate $i\in\{1,\ldots,n\}$, there is an automorphism of $B_n$ that interchanges the two coordinate halfspaces $V(B_n)\cap\{x_i=0\}~\text{and}~V(B_n)\cap\{x_i=1\}$.

First suppose $i\neq 5$. Let $\tau$ be the automorphism of $Q_n$ defined by
\[
\tau(x)_5=x_5,
\tau(x)_j=1-x_j~\text{for }j\neq 5.
\]
The map $\tau$ pairs the deleted vertices as follows:
\[
00000 0^k\leftrightarrow 11110 1^k,
~
00001 0^k\leftrightarrow 11111 1^k,
~
00010 0^k\leftrightarrow 11100 1^k,
~
00101 0^k\leftrightarrow 11011 1^k.
\]
Then, $\tau$ maps $S_n$ onto itself. Thus, $\tau$ restricts to an automorphism of $B_n$. Since $i\neq 5$, the map $\tau$ changes the $i$-th coordinate, and hence it interchanges the two $i$-coordinate halfspaces of $B_n$.

For $i=5$, let $\sigma$ be the automorphism of $Q_n$ obtained by interchanging coordinates $3$ and $4$ and, at the same time, complementing coordinate $5$; all other coordinates are fixed.  Explicitly, $\sigma(x)_3=x_4, \sigma(x)_4=x_3,\sigma(x)_5=1-x_5,$ and $\sigma(x)_j=x_j$ for $j\notin\{3,4,5\}$. A direct check gives
\[
00000 0^k\leftrightarrow 00001 0^k,
~
00010 0^k\leftrightarrow 00101 0^k,
~
11011 1^k\leftrightarrow 11100 1^k,
~
11110 1^k\leftrightarrow 11111 1^k.
\]
Hence, $\sigma(S_n)=S_n$, so $\sigma$ restricts to an automorphism of $B_n$. Since $\sigma$ changes the fifth coordinate, it interchanges the two fifth-coordinate halfspaces of $B_n$.

Thus, \(B_n\) is opposite-semicube-isomorphic.
\end{proof}

We now show that the graphs $B_n$ are not antipodal.

\begin{proposition}\label{prop:not-antipodal}
For every $n\geq 5$, the graph $B_n$ is not antipodal.
\end{proposition}

\begin{proof}
Let $k=n-5$ and denote $x=111011^k,~y=000110^k,~z=00110\,0^k.$  All of these three vertices belong to $V(B_n)$. Since $B_n$ is isometric in $Q_n$, distances in $B_n$ are the corresponding Hamming distances. In particular, $d_{B_n}(x,y)=d_{B_n}(x,z)=n-1.$

The vertices $x$ and $y$ agree only in coordinate $5$. The vertex
obtained from $y$ by changing coordinate $5$ is $y^{(5)}=000100^k,$ which belongs to $S_n$. Hence, every neighbor $y'$ of $y$ in $B_n$ is obtained by changing a coordinate in which $x$ and $y$ differ. Thus,
\[
        d_{B_n}(x,y')=n-2<n-1=d_{B_n}(x,y)
\]
for every neighbor $y'$ of $y$. Thus, $y$ is a relative antipode of $x$.

Similarly, the vertices $x$ and $z$ agree only in coordinate $3$, and $z^{(3)}=000100^k\in S_n.$ Therefore, every neighbor $z'$ of $z$ in $B_n$ is obtained by changing a coordinate in which $x$ and $z$ differ. Hence,
\[
        d_{B_n}(x,z')=n-2<n-1=d_{B_n}(x,z)
\]
for every neighbor $z'$ of $z$. Thus, $z$ is also a relative antipode of $x$.

Since $y\ne z$, the vertex $x$ has at least two distinct relative
antipodes. Therefore, $B_n$ is not antipodal.
\end{proof}

Combined with the preceding results, the main theorem is obtained as follows.

\begin{theorem}\label{thm:main1}
For every $n\geq 5$, the graph
\[
\begin{aligned}
B_n=Q_n-\{&00000 0^{n-5},00001 0^{n-5},00010 0^{n-5},00101 0^{n-5},\\
&11011 1^{n-5},11100 1^{n-5},11110 1^{n-5},11111 1^{n-5}\}
\end{aligned}
\]
is a partial cube with pairwise isomorphic opposite semicubes, but $B_n$ is not antipodal, and equivalently, not harmonic-even.
\end{theorem}

\section{Cartesian product extensions}

The family of graphs $\left\{B_n\right\}_{n=5}^{\infty}$ already gives infinitely many counterexamples. We record the product construction because it yields additional graph families.

\begin{lemma}\label{lem:product-semicubes2}
Let \(G\) and \(X\) be partial cubes. If both \(G\) and \(X\) are opposite-semicube-isomorphic, then \(Y=G\square X\) is also
opposite-semicube-isomorphic.
\end{lemma}

\begin{proof}
Let \(Y=G\square X\). Since the Cartesian product of partial cubes is a partial cube, \(Y\) is a partial cube.

First consider a \(\Theta\)-class \(E_i^G\) of \(G\). By Lemma~\ref{lem:product-semicubes}, the corresponding \(G\)-direction \(\Theta\)-class of \(Y\) has semicubes
\[
W_{G,i}^{Y,\varepsilon}=W_i^{G,\varepsilon}\times V(X)~\text{for}
~ \varepsilon\in\{0,1\}.
\]
Since \(G\) is opposite-semicube-isomorphic, we have $G[W_i^{G,0}]\cong G[W_i^{G,1}].$ Taking the Cartesian product with \(X\), we obtain $G[W_i^{G,0}]\square X \cong G[W_i^{G,1}]\square X.$ Moreover,
\[
Y[W_i^{G,\varepsilon}\times V(X)]
\cong
G[W_i^{G,\varepsilon}]\square X
~\text{for}~\varepsilon\in\{0,1\}.\]
Hence, $Y[W_{G,i}^{Y,0}]\cong
Y[W_{G,i}^{Y,1}].$

The argument for a \(\Theta\)-class \(E_j^X\) of \(X\) is analogous. Thus, for every \(\Theta\)-class of \(Y=G\square X\), the two opposite semicubes induce isomorphic subgraphs. Hence, \(Y=G\square X\) is opposite-semicube-isomorphic.
\end{proof}

\begin{theorem}\label{thm:product-counterexamples}
Let $G$ be an opposite-semicube-isomorphic partial cube that is not antipodal, and let $X$ be an antipodal opposite-semicube-isomorphic partial cube. Then $G\square X$ is an opposite-semicube-isomorphic partial cube that is not antipodal.
\end{theorem}

\begin{proof}
By Lemma~\ref{lem:product-semicubes2}, the Cartesian product $G\square X$ is an opposite-semicube-isomorphic partial cube. In what follows, it suffices to prove that $G\square X$ is not antipodal.

Because any vertex at maximum distance from a given vertex is a relative antipode of that vertex, the finiteness and connectedness of \(G\) imply that every vertex of \(G\) has at least one relative antipode. As \(G\) is not antipodal, there exists a vertex \(x\in V(G)\) having two distinct relative antipodes \(y,z\in V(G)\).

Choose an arbitrary vertex $t\in V(X)$, and let $\bar t$ be the unique relative antipode of $t$ in $X$. Define $a=(x,t),~u=(y,\bar t),~v=(z,\bar t).$ We show that both $u$ and $v$ are relative antipodes of $a$ in $G\square X$.

First, consider $u$. By the distance formula for Cartesian products,
\[
        d_{G\square X}(a,u)
        =d_G(x,y)+d_X(t,\bar t).
\]
Let $u'$ be an arbitrary neighbor of $u$ in $G\square X$. There are two possible cases.

Suppose first that $u'=(y',\bar t),$ where $y'$ is a neighbor of $y$ in $G$. Since $y$ is a relative antipode of $x$, the definition of a relative antipode gives $d_G(x,y)\ge d_G(x,y').$ Therefore,
\[
\begin{aligned}
d_{G\square X}(a,u')
   &=d_G(x,y')+d_X(t,\bar t)\\
   &\le d_G(x,y)+d_X(t,\bar t)\\
   &=d_{G\square X}(a,u).
\end{aligned}
\]

Suppose next that $u'=(y,t'),$ where $t'$ is a neighbor of $\bar t$ in $X$. Since $\bar t$ is a relative antipode of $t$, we have $d_X(t,\bar t)\ge d_X(t,t').$ Therefore,
\[
\begin{aligned}
d_{G\square X}(a,u')
   &=d_G(x,y)+d_X(t,t')\\
   &\le d_G(x,y)+d_X(t,\bar t)\\
   &=d_{G\square X}(a,u).
\end{aligned}
\]
Thus, $d_{G\square X}(a,u)\ge d_{G\square X}(a,u')$ for every neighbor $u'$ of $u$. Hence, $u$ is a relative antipode of $a$.

The same argument, with $z$ in place of $y$, shows that $v$ is also a relative antipode of $a$. Since $y\ne z$, we have $u=(y,\bar t)\ne (z,\bar t)=v.$ Therefore, the vertex $a$ has at least two distinct relative antipodes in $G\square X$. It follows that $G\square X$ is not antipodal.
\end{proof}
\begin{example}\label{exa:ExampleofCartesianProduct}
    Taking \(G=B_n\) and \(X=Q_m\) gives the family $G_{n,m}=B_n\square Q_m$ for $n\ge 5$ and $m\ge 0.$ By Theorem \ref{thm:product-counterexamples}, $G_{n,m}$ is an opposite-semicube-isomorphic partial cube that is not antipodal. Consequently, it is not harmonic-even. 

\end{example}

\section{A Helly-type condition forcing antipodality}

The counterexamples presented in the preceding sections demonstrate that local opposite-semicube isomorphisms---considered in isolation---are insufficient to guarantee antipodality. In this section, we establish that the opposite-semicube Helly property constitutes the precise additional condition required to ensure antipodality.


\begin{lemma}\label{lem:quadrant-balance}
Let $G$ be a finite opposite-semicube-isomorphic partial cube. Let $E_i$ and $E_j$ be two distinct $\Theta$-classes of $G$. Then, for every $\varepsilon,\delta\in\{0,1\}$,
\[
        |W_i^{\varepsilon}\cap W_j^{\delta}|
        =
        |W_i^{1-\varepsilon}\cap W_j^{1-\delta}|.
\]
\end{lemma}

\begin{proof}
Since $G$ is opposite-semicube-isomorphic, the two semicubes determined by each $\Theta$-class are isomorphic and hence have the same cardinality. Thus, $|W_i^{0}|=|W_i^{1}|,~|W_j^{0}|=|W_j^{1}|.$ Put
\[
        a_{\varepsilon\delta}
        =
        |W_i^{\varepsilon}\cap W_j^{\delta}|
        ~\text{for} ~\varepsilon,\delta\in\{0,1\}.
\]
Since $W_{i}^{0}=(W_{i}^{0}\cap W_{j}^{0})\cup (W_{i}^{0}\cap W_{j}^{1})$ and the two sets $(W_{i}^{0}\cap W_{j}^{0})$ and $(W_{i}^{0}\cap W_{j}^{1})$ are disjoint, we have $|W_{i}^{0}|=a_{00}+a_{01}$. Similarly, $|W_{i}^{1}|=a_{10}+a_{11}$, $|W_{j}^{0}|=a_{00}+a_{10}$, and $|W_{j}^{1}|=a_{01}+a_{11}$. Therefore, $a_{00}+a_{01}=a_{10}+a_{11},$ and $a_{00}+a_{10}=a_{01}+a_{11}.$ Subtracting the second equality from the first gives \(a_{01}=a_{10}\), and then either equality gives \(a_{00}=a_{11}\). This is exactly the desired assertion.
\end{proof}

\begin{lemma}\label{lem:opposite-family-pairwise}
Let $G$ be a finite opposite-semicube-isomorphic partial cube. Then, for every vertex $x\in V(G)$, the family
\[
        \mathcal O(x)=\{W_i^{*}(x):1\le i\le d\}
\]
is pairwise intersecting.
\end{lemma}

\begin{proof}
Let $i\ne j$. Since $x\in W_i(x)\cap W_j(x)$, we have $W_i(x)\cap W_j(x)\ne\emptyset .$ By Lemma~\ref{lem:quadrant-balance},
\[
        |W_i^*(x)\cap W_j^*(x)|
        =
        |W_i(x)\cap W_j(x)|.
\]
Hence,
\[
        W_i^*(x)\cap W_j^*(x)\ne\emptyset .
\]
Since this holds for every pair $i\ne j$, the family $\mathcal O(x)$ is pairwise intersecting.
\end{proof}

\begin{lemma}\label{lem:canonical-complement}
Let $G$ be a finite opposite-semicube-isomorphic partial cube with the opposite-semicube Helly property. Then, for every vertex $x\in V(G)$, there exists a unique vertex $ x'\in V(G)$ such that
$$
        x'\in \bigcap_{i=1}^d W_i^*(x).
$$
Equivalently, $x$ and $x'$ lie in opposite semicubes for every $\Theta$-class of $G$.
\end{lemma}

\begin{proof}
By Lemma~\ref{lem:opposite-family-pairwise}, the family $\mathcal O(x)=\{W_i^*(x):1\le i\le d\}$ is pairwise intersecting. Since $G$ has the opposite-semicube Helly property, we obtain $\bigcap\limits_{i=1}^d W_i^*(x)\ne\emptyset .$ Thus, there exists at least one such vertex $x'$.

Suppose that $u$ and $v$ both belong to $\bigcap\limits_{i=1}^d W_i^*(x).$ Then, for every $\Theta$-class $E_i$, the vertices $u$ and $v$ lie in the same semicube determined by $E_i$. Hence, no $\Theta$-class separates $u$ and $v$. Since, in a partial cube, the distance between two vertices equals the number
of $\Theta$-classes that separate them, we get $d_G(u,v)=0.$ Therefore, $u=v$.
\end{proof}

\begin{theorem}\label{thm:OH-OSI-antipodal}
Let \(G\) be a finite partial cube with the opposite-semicube Helly property. Then \(G\) is antipodal if and only if \(G\) is opposite-semicube-isomorphic.
\end{theorem}

\begin{proof}
Let \(G\) be a finite partial cube and $E_{1}$, \ldots, $E_{d}$ be the $\Theta$-classes of $G$. Suppose first that \(G\) is antipodal. By Theorem~\ref{T2.1}, \(G\) admits a semicube-switching automorphism. Hence, for every $\Theta$-class \(E_{i}\), this automorphism restricts to an isomorphism $G[W^0_{i}]\cong G[W^1_{i}].$ Thus, \(G\) is opposite-semicube-isomorphic.

Conversely, suppose that $G$ is opposite-semicube-isomorphic and has the opposite-semicube Helly property. By Lemma~\ref{lem:canonical-complement}, for every vertex $x\in V(G)$, there exists a unique vertex $x'\in V(G)$ lying in the semicube opposite to $x$ for every $\Theta$-class.

We prove that $x'$ is the unique relative antipode of $x$. Let $u$ be a neighbor of $x'$, and let $E_i$ be the $\Theta$-class containing the edge $x'u$. The vertices $x$ and $x'$ are separated by every $\Theta$-class. On the other hand, $x$ and $u$ are separated by every $\Theta$-class except $E_i$. Since distance in a partial cube equals the number of $\Theta$-classes separating the two vertices, we obtain
\[
        d_G(x,u)=d_G(x,x')-1.
\]
Thus, $d_G(x,x')>d_G(x,u)$ for every neighbor $u$ of $x'$. Hence, $x'$ is a relative antipode of $x$.

Let $y\ne x'$, and choose a shortest $y,x'$-path $\langle y=y_0,y_1,\ldots,y_r=x'\rangle.$ Let $E_j$ be the $\Theta$-class containing the edge $y_0y_1$. Since $E_j$ separates $y$ from $x'$, and $x'$ lies in the semicube opposite to $x$ with respect to $E_j$, the vertices $x$ and $y$ lie in the same semicube determined by $E_j$. After crossing the edge $y_0y_1$, the vertices $x$ and $y_1$ lie in opposite semicubes with respect to $E_j$, while whether $x$ and the current vertex are separated by any other $\Theta$-class remains unchanged. Therefore,
\[
        d_G(x,y_1)=d_G(x,y)+1.
\]
Since $y_1$ is a neighbor of $y$, the vertex $y$ is not a relative
antipode of $x$. Therefore, $x'$ is the unique relative antipode
of $x$, and hence $G$ is antipodal.
\end{proof}

Combining Theorems \ref{T2.1} and \ref{thm:OH-OSI-antipodal}, we derive the following corollary.

\begin{corollary}
Let \(G\) be a finite partial cube with the opposite-semicube Helly property. Then \(G\) is harmonic-even if and only if \(G\) is opposite-semicube-isomorphic.
\end{corollary}


The failure of the opposite-semicube Helly property in \(B_n\) can be seen explicitly. Let $z=111011^{n-5}\in V(B_n).$ In the ambient hypercube \(Q_n\), the unique vertex lying in the semicube opposite to \(z\) in every coordinate direction is $000100^{n-5}.$ This vertex belongs to \(S_n\), and therefore it is not a vertex of \(B_n\). Hence, inside \(B_n\), $\bigcap_{i=1}^n W_i^{*}(z)=\emptyset.$ On the other hand, \(B_n\) is opposite-semicube-isomorphic by Proposition~\ref{prop:semicubes}. Thus, Lemma~\ref{lem:opposite-family-pairwise} implies that the family $\mathcal O_{B_n}(z)=\{W_i^{*}(z)\mid 1\le i\le n\}$ is pairwise intersecting. Therefore, for the vertex \(z\), the semicubes opposite to \(z\) are pairwise intersecting but have empty total intersection.

The example above shows concretely how the Helly-type condition fails in \(B_n\). Pairwise isomorphism of opposite semicubes provides local balance, but it does not guarantee the existence of a common vertex lying in all semicubes opposite to a fixed vertex.
\enlargethispage{2\baselineskip}

\section{Concluding remarks}

Our results reveal two complementary aspects of the local isomorphism condition on opposite semicubes. On the one hand, we give a negative answer to the question of Klav\v{z}ar and Kov\v{s}e. The graphs $B_n=Q_n-S_n$, $n\ge 5$, are partial cubes with pairwise isomorphic opposite semicubes, but they are not antipodal (hence not harmonic-even by Polat's theorem). Equivalently, their local opposite-semicube isomorphisms cannot be assembled into a global antipodal structure.

On the other hand, we show that the opposite-semicube Helly property is exactly such a condition to ensure the equivalence between opposite-semicube isomorphism and antipodality, i.e., a finite partial cube with this Helly-type property is antipodal if and only if its opposite semicubes are pairwise isomorphic. Moreover, the Helly condition is essential, as the counterexample family $B_n$ precisely violates it.

The base graph
\[
B_5=Q_5-\{00000,00001,00010,00101,11011,11100,11110,11111\}
\]
has $24$ vertices. It is the smallest counterexample found in our present search. However, our arguments do not prove that it is order-minimal. This leads to the following problem.

\begin{problem}
Let $E_{1}$, \ldots, $E_{d}$ be the $\Theta$-classes of partial cube $G$. Determine the minimum order of $G$ satisfying $G[W^0_{i}]\cong G[W^1_{i}]~\text{for every}~\Theta \text{-class of}~G$, but which is not antipodal or harmonic-even. Is the graph $B_5$ order-minimal? If so, can one classify all counterexamples on $24$ vertices?
\end{problem}

A second problem is to find weaker structural assumptions than the opposite-semicube Helly property that still allow local opposite-semicube isomorphisms to be lifted to antipodality or harmonic-evenness.

\begin{problem}
Find natural subclasses of finite partial cubes in which
opposite-semicube-isomorp\\hism implies the opposite-semicube Helly property, and hence implies antipodality or harmonic-evenness.
\end{problem}

\section{Declaration of competing interest}
The authors declare that they have no known competing financial interests or personal relationships that could have appeared to influence the work reported in this paper.

\section{Data availability}
No data was used for the research described in the article.

\section{Acknowledgments}
This work was partially supported by the National Natural Science Foundation of China (No. 12071194).

\end{CJK}

\begin{thebibliography}{99}


\bibitem{BandeltChepoi2008}
H.-J. Bandelt and V. Chepoi,
Metric graph theory and geometry: a survey,
in: J. E. Goodman, J. Pach and R. Pollack (eds.),
\emph{Surveys on Discrete and Computational Geometry: Twenty Years Later},
Contemporary Mathematics, Vol. 453, American Mathematical Society,
Providence, RI, 2008, pp. 49--86.




\bibitem{Djokovic1973}
D. Djokovi\'c,
Distance preserving subgraphs of hypercubes,
\emph{J. Combin. Theory Ser. B} 14 (1973), 263--267.



\bibitem{FukudaHanda1993}
K. Fukuda and K. Handa, Antipodal graphs and oriented matroids, \emph{Discrete Math.} 111 (1993), 245--256.


\bibitem{Glivjak1970} F. Glivjak, A. Kotzig and J. Plesn\'{i}k, Remark on the graphs with a central symmetry, \emph{Monatsh. Math.} 74 (1970), 302--307.


\bibitem{Graham1971} R. L. Graham and H. Pollak, On the addressing problem for loop switching, \emph{Bell Syst. Tech. J.} 50 (1971), 2495--2519.

\bibitem{HammackImrichKlavzar2011}
R. Hammack, W. Imrich and S. Klav\v{z}ar,
Handbook of product graphs, second edition,
CRC Press, 2011.

\bibitem{Handa1999} K. Handa, Bipartite graphs with balanced $(a,b)$-partitions, \emph{Ars Combin.} 51 (1999), 113--119.

\bibitem{Imrich2000} W. Imrich and S. Klav\v{z}ar, Product graphs: structure and recognition, \emph{John Wiley $\&$ Sons}, 2000.


\bibitem{KlavzarKovse2009}
S. Klav\v{z}ar and M. Kov\v{s}e,
On even and harmonic-even partial cubes,
\emph{Ars Combin.} 93 (2009), 77--86.

\bibitem{Kotzig1968} A. Kotzig, Centrally symmetric graphs, \emph{Czechoslovak Math. J.} 18 (1968), 606--615.




\bibitem{MulderSchrijver1979}
H. M. Mulder and A. Schrijver,
Median graphs and Helly hypergraphs,
\emph{Discrete Math.} 25 (1979), 41--50.


\bibitem{Ovchinnikov2008} S. Ovchinnikov, Partial cubes: structures, characterizations, and constructions, \emph{Discrete Math.} 308 (2008), 5597--5621.


\bibitem{Polat2019}
N. Polat,
On some characterizations of antipodal partial cubes,
\emph{Discuss. Math. Graph Theory} 39 (2019), 439--453.

\bibitem{Sabidussi1996} G. Sabidussi, Graphs without dead ends, \emph{European J. Combin.} 17 (1996), 69--87.


\bibitem{Winkler1984}
P. Winkler,
Isometric embeddings in products of complete graphs,
\emph{Discrete Appl. Math.} 7 (1984), 221--225.

\end{thebibliography}
\end{document}